\documentclass[a4paper,11pt]{article}

\title{{\bf On the Nearest Neighbor Algorithm for Mean Field Traveling Salesman Problem}}
\author{ {\bf Antar Bandyopadhyay}\footnote{E-Mail: antar@isid.ac.in} \\  {\bf Farkhondeh Sajadi}\footnote{E-Mail: farkhondeh.sajadi@gmail.com}\\  \\
\small Theoretical Statistics and Mathematics Unit,\\
\small Indian Statistical Institute, Delhi Centre,\\
\small 7 S. J. S. Sansanwal Marg \\ New Delhi 110016 \\ INDIA}
\date{}
\RequirePackage[OT1]{fontenc}
\RequirePackage{amsthm,amsmath}
\RequirePackage[numbers]{natbib}
\RequirePackage[colorlinks,citecolor=blue,urlcolor=blue]{hyperref}
\usepackage{setspace,graphics,chngpage,amsthm}
\usepackage{amsmath, array,amsfonts,amssymb,newlfont,hhline,setspace}
\usepackage{fullpage}
\usepackage{graphicx,bm}

\DeclareOldFontCommand{\rm}{\rmfamily}{\mathrm}
\newtheorem{theorem}{Theorem}[section]
\newtheorem{lemma}{Lemma}[section]
\newtheorem{proposition}{Proposition}[section]
\newtheorem{corollary}{Corollary}[section]

\theoremstyle{definition}

\theoremstyle{remark}

\numberwithin{equation}{section}
\addtocounter{tocdepth}{1}

\newcommand{\TNN}{T^{NN}_{n}}
\newcommand{\bE}{\mathbb{E}}
\newcommand{\bP}{{\mathbf P}}
\newcommand{\Lt}{\mathcal{L}_{2}}
\newcommand{\Pro}{\mathbb{P}}
\newcommand{\Var}{\mathsf{Var}}

\newcommand{\bone}{\mathbf{1}}

\begin{document}

\maketitle

\begin{abstract}
In this work we consider the mean field  traveling salesman problem, where the intercity distances are taken to be i.i.d. with some distribution $F$.  
This paper focus on the \emph{nearest neighbor tour}  which is to move to the nearest non-visited city and we show that under some conditions on $F$, 
which are satisfied by exponential distribution with constant mean, 
the total length of the nearest neighbor tour, asymptotically almost surely scales as $\log n$. Similar result is known for Euclidean TSP and nearest 
neighbor tour. We further derive the limiting behavior of the total length of the nearest neighbor tour for more general distribution function $F$ 
and show that its asymptotic properties are determined by the scaling properties of the density of $F$ at $0$.
\end{abstract}

\vspace{0.25in}

\vspace{0.25in}

\section{Introduction}
\label{Intro}
The traveling salesman problem (TSP) is a very well known combinatorial optimization problem. The aim is to find the shortest tour, 
connecting a number of cities visited by a traveling salesman on his sales route, such that he visits each city exactly once and finally returns 
to the starting city. Formally, we are given a set $\left\{c_{1}, c_{2}, \ldots, c_{n}\right\}$ of \emph{cities} and for each pair 
$\left\{c_{i}, c_{j}\right\}$ of distinct cities, a distance $d(c_{i}, c_{j})$. The goal is to find a permutation $\pi$ of the cities that minimizes 
the quantity
\begin{equation}\sum_{i=1}^{n}d(c_{\pi(i)}, c_{\pi(i+1)}) \label{TSP} \end{equation}
where $\pi(n+1)=1$. This quantity is called the \emph{tour length}, since it is the total distance traveled by the salesman. We shall 
concentrate in this chapter on the \emph{symmetric} TSP, in which the distances satisfy 
\[
d(c_{i}, c_{j}) = d(c_{j}, c_{i}) \quad \text{for} \quad 1 \leq i, j \leq n.
\]

There are several randomized versions of this problem where the distances are taken to be random. In particular the one 
which attracted considerable attention among mathematicians and computer scientists is known as the \textit{Euclidean TSP}, 
in which the $n$ cities are randomly distributed in a $d$-dimensional hypercube and the distances between cities are given by the Euclidean metric and
are thus random.
The other random TSP, which has been of interest within the statistical physics community is \textit{the mean field TSP}. Here the distances 
between pairs of cities, i.e., $d(c_{i}, c_{j})$ are taken as independent random variables with a given distribution $F$. 
Note that in this case, the geometric structure may break since the triangle inequality may not necessarily hold with probability one. In fact we 
cannot quite say that the numbers $d(c_{i}, c_{j})$ really represent distances under any metric. Although this seems artificial, however 
such models are of interest in statistical physics literature.  

It is well known in algorithm literature \cite{PaSt98} that TSP in general is a \emph{NP-Complete} problem. So there are several approximate algorithms
which tries to approximate the optimal tour with polynomial running time. Among them, one of the simplest is 
the \emph{Nearest Neighbor (NN) Algorithm} \cite{BeNe68}, which is also known as \emph{the next best method} \cite{Ga65}.
It was one of the first algorithms used to determine an approximate solution to the traveling salesman problem. 
The algorithm starts with a tour containing a randomly chosen city and then always adds the nearest not yet
visited city to the last city in the tour. The algorithm terminates when every city has been added to the tour.
In the NN algorithm, a tour is constructed as follows:

\begin{flushleft}

\begin{verse} 

\begin{itemize}
{\tt
\item[Step-0:] Input graph $G$ with a linear ordering of its vertices say 
              \[V := \left\{c_{1}, c_{2}, \ldots, c_{n}\right\}.\] Let
              $Tour \leftarrow \left\{c_{1}\right\}$ and $c_{\pi(1)}=c_1$.\\
              
\item[Step-1: ] Write $Tour \leftarrow \left\{c_{\pi(1)}, c_{\pi(2)}, \ldots, c_{\pi(i)}\right\}$. Choose $c_{\pi(i+1)}$ to be the city $c_j$ that minimizes \[\lbrace d(c_{\pi(i)} ,c_j) : j \neq \pi(k) , 1\leq k \leq i\rbrace .\]
 Update $Tour$ as
  \[   Tour \leftarrow Tour \cup \left\{c_{\pi(i+1)}\right\}. \]\\

\item[Step-2: ] Go to Step-1 unless  $V \setminus Tour=\emptyset$.\\

\item[Step-3: ] Stop with output $Tour$ as the NN tour with starting city $c_1$. \\
}
\end{itemize}
\end{verse}
\end{flushleft}

For the convenience, when there are ties in {\tt Step-1}, we assume that they can be 
broken arbitrarily. The NN algorithm can be improved by repeating the algorithm for each possible starting city and then take the minimum solution among them 
\cite{Ga65}. It is known that, for TSP on $n$ cities, the running time for NN  
algorithm is $O(n^2)$ \cite{JoMcGe97, RoStLe77}.

Denote the distance $d(c_i,c_j)$ by $L_{ij}$. Since the NN algorithm is to move to the nearest non-visited city, therefore starting from $c_1$, by using this 
algorithm we need to find the nearest city to it. We call it $v_{2}$. In this way, we need to find 
\[
\displaystyle \min\left\{L_{12},L_{13},\ldots,L_{1n}\right\}
\]
Then from city $v_{2}$ we find the nearest city to that and call it $v_{3}$. Here we need to find 
\[
\displaystyle \min\left\{L_{v_{2}u} |  u \in \left\{2, 3, \ldots, n\right\} \quad \text{and} \quad u\neq v_{2} \right\}.
\]
We continue the algorithm till all $n$ cities have been visited. Then from there we go back to starting city which is $c_1$.

Define $\TNN$ to be the length of NN tour among $n$ cities in the TSP, then
\begin{equation} 
\TNN = \sum_{i=1}^{n} L_{v_{i}v_{i+1}}, \quad  v_{1}=1=v_{n+1}  \label{Tn} \,.
\end{equation} 

\subsection{The deterministic TSP }
The performance of nearest neighbor algorithm has been studied for the TSP when the distances are defined through a metric. 
Let $T^{opt}_n$ be the length of the optimal tour and $\left\lceil x\right\rceil$ denote the smallest integer greater than or equal to $x$. 
\cite{RoStLe77} measured the closeness of a tour by the ratio of the obtained tour length, to the optimal tour length. 
They proved that if the cities are placed in a metric space and the intercity distances are given by the metric then
\[
\frac{\TNN}{T^{opt}_n} \leq \frac{1}{2}\left\lceil \log_2 n\right\rceil+\frac{1}{2} \,.
\]
They also showed that for each $m>3$, there exists a traveling salesman graph with $n=2^m-1$ nodes inside a metric space such that
\[\frac{\TNN}{T^{opt}_n} > \frac{1}{3} \log_2(n+1) +\frac{4}{9} \,.\]
\subsection{The random TSP }
One of the famous mathematical results for the Euclidean TSP is Beardwood-Halton-Hammersley theorem which studies the large sample behavior of the length of 
shortest tour in TSP.
Let the cities be independently and uniformly distributed on $[0,1]^{d}$.
\cite{BeHaHa59} showed that there is a constant $0<\beta_{TSP}(d)<\infty$ such that with probability one
\[
\frac{T^{opt}_{n}}{n^{\frac{d-1}{d}}} \longrightarrow \beta_{TSP}(d)
\]
They also proved that for nonuniform random samples, there is an universal constant $\beta_{TSP}(d)$ such that 
\[
\frac{T^{opt}_{n}}{n^{\frac{d-1}{d}}} \longrightarrow \beta_{TSP}(d) \int_{\mathbb{R}^{d}} f(x)^{(d-1)/d} dx \quad \mbox{a.s.}
\]
where $f(x)$ is the density of the absolutely continuous part of the distribution of cities with a compact support.

Asymptotic results in the mean field TSP have been obtained by \cite{Wast10}. Let  $L_{ij}$'s be independent random variables from a fixed distribution on the 
nonnegative real numbers. Suppose as $t\longrightarrow 0^{+}$
\[
\frac{\Pro(L_{ij}< t) }{t} \longrightarrow 1 
\] 
He proved that for large $n$, 
\begin{equation}
T^{opt}_{n} \stackrel{\Pro}{\longrightarrow} \frac{1}{2} \int_{0}^{\infty} \! h(x) \,dx
\label{Equ:Wastlund}
\end{equation}
where $h$ as a function of $x$ is implicitly defined through the equation
\[
\left(1+\frac{x}{2}\right)e^{-x} + \left(1+\frac{h(x)}{2}\right)e^{-h(x)} = 1
\]
Although there seems to be no simple expression for this limit in terms of known mathematical constants, it can be evaluated numerically to be approximately 
$ 2.041548$.

In this paper we study the limiting behavior of the total length of the tour, obtained by NN algorithm for the mean field TSP. 
Our motivation is similar to that of \cite{RoStLe77}. We would like 
to compare the apparent ``loss'' (that is, more distance to be traversed) accrued by using the NN algorithm with respect to the optimal solution.
But because of \eqref{Equ:Wastlund}, it is enough to consider the limiting behavior of $\TNN$.
We show if $F$, the distribution of the distance between cities, has a density which is continuous at $0$ with $F'\left(0+\right) > 0$, 
then the total length of the NN tour for mean field TSP scales as $\log n$. This parallels the conclusions drawn in \cite{RoStLe77} for
Euclidean TSP.
Moreover we also consider a general distribution function $F$ with non-negative support and show that the asymptotic behaviors for 
$\TNN$ depend on the limiting 
properties of the density near $0$. 

The rest of the paper is structured as follows. In the following section we state our main results whose proofs are given in Section~\ref{Proofs}.
In Section~\ref{Sec:Aux} we present three auxiliary results and their proofs which we need in proving the main results. 
Section~\ref{last edge} contains a study the first and the last edges of NN tour in the mean field TSP and we show that the sum total of the
first and last edge weights remains tight as the number of cities grow to infinity. Finally, 
in Section~\ref{discuss} we discuss about possible relaxation of the assumptions on the distribution $F$.

\section{Main results}
\label{Main results}
We will assume that the mean and the variance of $F$ are finite and $F$ has a density $f$. 
Our first result shows that $\TNN$ is ``close'' to its expected value. 
\begin{theorem}
\label{TNN-E(TNN) converges}
Assume that as $t\longrightarrow 0+,~ \frac{f(t)}{t^{\alpha}}\longrightarrow C$, 
where $C\in (0,\infty)$ is constant and  $-1< \alpha < 1$.  Then as $n\longrightarrow \infty,$ 
\begin{equation}
\lbrace \TNN-\bE[\TNN]\rbrace_{n\geq1} \quad  \text{converges weakly}. 
 \label{TNN-E(TNN)} \end{equation}
\end{theorem}

The three main results of the paper consider three cases of the behavior of $f$ near $0$. Theorem~\ref{Thm f(t)-->f(o)} 
covers the case when $f$ near zero converges to a constant. In this case, $\TNN$ scales as constant times $\log n$. 
Theorem~\ref{Thm f(t)/t^alpha, 0<alpha<1} and Theorem~\ref{Thm f(t)/t^alpha, -1<alpha<0} consider the cases when
$\lim_{t \rightarrow 0} f(t)$ is zero and infinity respectively.
We use the notation $a_{n} \sim b_{n}$ to denote $a_{n}$ is asymptotically equal to $b_{n}$, that is, 
$\displaystyle \lim_{n \longrightarrow \infty} \frac{a_{n}}{b_{n}}=1.$
\begin{theorem}
\label{Thm f(t)-->f(o)}
Assume that as $t\longrightarrow 0+, ~f(t)\longrightarrow f(0)$, 
where $f(0) \in (0,\infty)$. Then as $n\longrightarrow \infty,$ 
\begin{equation}
\frac{\TNN}{\log n} \stackrel{\Pro}{\longrightarrow} \frac{1}{f(0)} \label{Tn/logn-->1/f(o)}
\end{equation}
and 
\begin{equation}
\bE[\TNN] \sim \frac{1}{f(0)}\log n  \label{E(Tn)=o(logn/f(o))}.
\end{equation}
Moreover, convergence in \eqref{Tn/logn-->1/f(o)} happens in $\Lt$.
\end{theorem}
When the distribution $F$ is Exponential, the expected value of the length of NN tour among $n$ cities scales as $\log n$. This is a special case of Theorem~\ref{Thm f(t)-->f(o)}, when $f(0)=1$. The following corollary 
is a consequence of Theorem~\ref{Thm f(t)-->f(o)}.
\begin{corollary}
\label{Thm for iid exp}
In the mean field TSP, suppose $F$ is the Exponential distribution with mean one. Then $\TNN-\log n$ converges weakly.  
\end{corollary}

\begin{theorem}
\label{Thm f(t)/t^alpha, 0<alpha<1}
Assume that as $t\longrightarrow 0+,~\frac{f(t)}{t^{\alpha}}\longrightarrow C$,
where $C>0$ is constant and  $0<\alpha <1$. Then as $n\longrightarrow \infty$, 
\begin{equation}
\frac{\TNN} {n^{1-\frac{1}{1+\alpha}}}\stackrel{\Pro}{\longrightarrow}K_{\alpha} \label{Tn/n^-->C}
\end{equation}
where
\[K_{\alpha}:=(\frac{1+\alpha}{C})^{\frac{1}{1+\alpha}} \frac{1+\alpha}{\alpha} \Gamma(1+\frac{1}{1+\alpha})\]
and
\begin{equation}
 \bE[\TNN] \sim K_{\alpha} n^{1-\frac{1}{1+\alpha}}  \label{E(Tn)=o(C)n}
 \end{equation}
 Moreover, convergence in \eqref{Tn/n^-->C} happens in $\Lt$.
\end{theorem}

\begin{theorem}
\label{Thm f(t)/t^alpha, -1<alpha<0}
Let $-1<\alpha<0$ and assume that as $t\longrightarrow 0+,\,\frac{f(t)}{t^{\alpha}}\longrightarrow C$,
where $C>0$ is constant. Then the sequence $\lbrace\bE[\TNN]\rbrace_{n\geq1}$, is a convergent sequence 
and $\TNN$ converges weakly.
\end{theorem}

The above results cover the cases where $\vert \alpha \vert< 1$. Note that the case $\alpha \leq -1$ cannot happen, since $f$ is a 
density function. For  $\alpha \geq 1$ we do not have any general result except for the particular choice of $F$, namely when 
$F$ is Weibull distribution with shape parameter $(1+\alpha)$ and scale parameter $1$, 
we show in the following theorem that after proper scaling,  the weak limit distribution of $\TNN$ is Normal. 
\begin{theorem}
\label{Prop f(t)/t^alpha,alpha>=1}
Let  $\alpha \geq 1$ and for $1\leq i \leq n-1$, the intercity distances $\left\{L_{ij}\right\}_{i< j \leq n }$  in mean field TSP be i.i.d. Weibull distribution with shape parameter $(1+\alpha)$ and scale parameter $1$, i.e.,
\[f(t)=(1+\alpha)t^\alpha e^{-t^{(1+\alpha)}}\bone\left(t>0\right)\,.\]
Then as $n\longrightarrow \infty $, for $\alpha >1$
\begin{equation}
\frac{\TNN-\bE[\TNN]}{n^{\frac{1}{2}-\frac{1}{1+\alpha}}}\stackrel{d}{\longrightarrow}  N( 0,\,\frac{\alpha+1}{\alpha-1} \sigma^2(\alpha))  
\label{Tn/n has Normal dist}\end{equation}
and for $\alpha=1$, 
\begin{equation}
\frac{\TNN-\bE[\TNN]}{\sqrt{\log n}}\stackrel{d}{\longrightarrow}  N(0,\sigma^2(\alpha)) 
\label{Tn/logn has Normal dist}\end{equation}
where $\sigma^2(\alpha)=\Gamma(\frac{2}{1+\alpha}+1)-\Gamma^{2}(1+\frac{1}{1+\alpha})$.
\end{theorem}

\section{The last and the first edges of the NN tour}
\label{last edge}
Let the distances between cities be denoted by $\left\{(L_{ij})_{i< j \leq n}\right\}_{1\leq i \leq n-1}$ which are i.i.d with distribution 
$F$ supported on $[0,\infty)$ and density $f$. 
Let $L_n^{\mbox{last}}$ be the length of the last edge, which joins the last visited city to the first city. Then the length of NN tour, $\TNN$,  
can be written as
\begin{equation} 
\TNN \stackrel{d}{=} \sum_{i=1}^{n-1} \displaystyle \min_{{i < j \leq n}}L_{ij} + L_n^{\mbox{last}} \label{Tndold} 
\end{equation}
Let $L_n^{\mbox{first}}:=\displaystyle \min_{{1 < j \leq n}}L_{1j}$. Then \eqref{Tndold} can be rewritten as, 
\begin{equation} 
\TNN \stackrel{d}{=} \sum_{i=2}^{n-1} \displaystyle \min_{{i < j \leq n}}L_{ij} + L_n^{\mbox{first}} + L_n^{\mbox{last}} \label{Tnd} 
\end{equation}
The following proposition shows that the sum of the lengths of the last and first edges in NN tour do not play an important role.
\begin{proposition}
\label{Prop1}
In the NN tour for mean field TSP, the distribution function of $L_n^{\mbox{first}} + L_n^{\mbox{last}}$ converges to $F$ as $n \longrightarrow \infty$ and $\displaystyle\sum_{i=2}^{n-1} \min_{i < j \leq n} L_{ij}$ is independent of $L_n^{\mbox{first}} + L_n^{\mbox{last}}$. 
Moreover as $n \longrightarrow \infty$, 
\[ \bE\left[ L_n^{\mbox{first}} + L_n^{\mbox{last}}\right] \longrightarrow \mu \]
and 
\[\bE\left[ \left( L_n^{\mbox{first}} + L_n^{\mbox{last}} \right)^2\right] \longrightarrow \mu^2+\sigma^2 \,,\] 
where $\mu$ and $\sigma^2$ are the mean and the variance of $F$. 
\end{proposition}
 
\begin{proof} 
For $k=1,2, \ldots, n-1$, let $X_k:= L_{1k+1}$ and $X_{(k)}$ be the $k^{th}$ order statistic of $X_1, X_2, \ldots, X_{n-1}$. Note that by assumption
$X_k$'s are i.i.d. $F$.\\
Notice that by construction the successive vertices $1 = v_1, v_2, v_3, \ldots, v_n$ of the tour have the property that for every $2 \leq k \leq n$ given 
$\left\{v_2, v_3, \cdots, v_{k-1}\right\}$ the vertex $v_k$ is uniformly distributed on the set 
$\left\{1, 2, \ldots, n\right\} \setminus \left\{1, v_2, v_3, \cdots, v_{k-1}\right\}$. Thus for every $3 \leq k \leq n$ given $v_2$,  the vertex $v_k$
is uniformly distributed on the set $\left\{2, 3, \ldots, n\right\} \setminus \left\{v_2\right\}$. So in particular the last vertex of the tour $v_n$
is also uniformly distributed on the set $\left\{2, 3, \ldots, n\right\} \setminus \left\{v_2\right\}$.
Hence given $X_1, X_2, \ldots, X_{n-1}$, the length of the last edge is uniform on $\left\{ X_{(2)}, X_{(3)}, \ldots, X_{(n-1)}\right\}$. 
Now for any  bounded continuous function $h$ we have, 
\begin{align*}
\bE\left[h\left( L_n^{\mbox{last}}\right)  \right]  & = \frac{1}{n-2} \sum_{k=2}^{n-1} \bE\left[h\left(X_{(k)}\right)  \right]\\
& = \frac{1}{n-2} \sum_{k=1}^{n-1} \bE\left[h\left(X_{(k)}\right)  \right] - \frac{\bE\left[h\left(X_{(1)}\right)  \right]}{n-2} \\
&=\frac{1}{n-2} \sum_{k=1}^{n-1} \bE\left[h\left(X_k\right)  \right] - \frac{\bE\left[h\left(X_{(1)}\right)  \right]}{n-2}\\
&=\frac{n-1}{n-2} \bE\left[h\left(X_1\right)  \right] - \frac{\bE\left[h\left(X_{(1)}\right)  \right]}{n-2}.
\end{align*} 
Therefore 
\[
\lim_{n \longrightarrow\infty}\bE\left[h\left( L_n^{\mbox{last}}\right)  \right]=\bE\left[h\left(X_1\right)  \right] \,,
\]
for every bounded continuous function $h$, thus the distribution function of
$L_n^{\mbox{last}} $ converges to $F$ as $n \longrightarrow \infty$.
Now observe that $L_n^{\mbox{first}} \longrightarrow 0$ almost surely, so by Slutsky's theorem we have the distribution function of 
$L_n^{\mbox{first}} + L_n^{\mbox{last}} $ converges to $F$ as $n \longrightarrow \infty$.

Now observe that by similar calculations as above
\[
\bE\left[ L_n^{\mbox{first}} + L_n^{\mbox{last}} \right] = \frac{n-1}{n-2} \bE\left[X_1\right] 
                                                           + \frac{n-3}{n-2} \bE\left[X_{(1)}\right] \longrightarrow \mu \,.
\]
The last limit follows from the dominated convergence theorem by observing that $X_{(1)} \longrightarrow 0$ almost surely and $0 \leq X_{(1)} \leq X_1$.

Further, 
\[
\bE\left[ \left(L_n^{\mbox{last}}\right)^2 \right] = \frac{n-1}{n-2} \bE\left[X_1^2 \right] - \frac{\bE\left[X_{(1)}^2\right]}{n-2} \longrightarrow \mu^2 + \sigma^2 \,,
\]
and
\[
\bE\left[ \left(L_n^{\mbox{first}}\right)^2 \right] = \bE\left[X_{(1)}^2\right] \longrightarrow 0 \,.
\]
Finally,
\begin{align*}
       \bE\left[ L_n^{\mbox{first}} L_n^{\mbox{last}} \right]
& =    \frac{n-1}{n-2} \bE\left[ X_{(1)} \bar{X}_{n-1} \right] - \frac{\bE\left[X_{(1)}^2\right]}{n-2} 
       \qquad \left[\mbox{where\ } \bar{X}_{n-1} := \frac{1}{n-1} \sum_{k=1}^{n-1} X_k \right]\\
& \leq \sqrt{\bE\left[ X_{(1)}^2 \right] \, \bE\left[ \bar{X}_{n-1}^2 \right]} - \frac{\bE\left[X_{(1)}^2\right]}{n-2}
       \qquad \left[\mbox{using Cauchy-Schwarz inequality}\right] \\
& =    \sqrt{\bE\left[ X_{(1)}^2 \right] \, \left(\mu^2 + \frac{\sigma^2}{n-1} \right)} - \frac{\bE\left[X_{(1)}^2\right]}{n-2} \\
& \longrightarrow 0 \,. 
\end{align*}
Combining all these we have
\[
\bE\left[ \left( L_n^{\mbox{first}} + L_n^{\mbox{last}} \right)^2\right] \longrightarrow \mu^2+\sigma^2 \,.
\]
\end{proof}

\section{Auxiliary results}
\label{Sec:Aux}
For the distribution function $F$ we define $F^{-1}:\left(0,1\right) \rightarrow [0,\infty)$ by 
$F^{-1}\left(u\right) := \inf \left\{ x \in \mathbb{R} \,\Big\vert\, F(x) \geq u \,\right\}$, $0 < u < 1$.
It is then a standard fact that $F^{-1}\left(U\right) \sim F$ when $U \sim \mbox{Uniform}\left[0,1\right]$.
We start with a lemma which will give an useful representation of $\TNN$.  
\begin{lemma}
\label{min L= min W}
Let the distances between cities, $(L_{ij})_{i < j \leq n}$ for $i=1,\ldots,n-1$ be i.i.d with $F$ denoting its common distribution function. 
Define the random variable $\displaystyle W_{i}:=F^{-1}\left( 1-\exp({-\frac{Y_{i}}{i}})\right) $ where $\left\{Y_{i}\right\}_{1 \leq i \leq n-1}$ are i.i.d. Exponential random variable each with mean one. Then
\[ \sum_{i=2}^{n-1} \displaystyle \min_{{i < j \leq n}}L_{ij} \stackrel{d}{=} \sum_{i=1}^{n-2}W_{i} \,.\]
Thus
\begin{equation} 
\TNN \stackrel{d}{=} \sum_{i=1}^{n-2}W_{i} + R_n \,, 
\label{Tnd=W} 
\end{equation}
where $R_n \stackrel{d}{=} L_n^{\mbox{first}} + L_n^{\mbox{last}}$ and is independent of $\left\{W_i\right\}_{i=1}^{n-2}$. 
\end{lemma}

\begin{proof}
Let $(\xi_{ij})_{i < j \leq n}$ be i.i.d. Exponential random variable each with mean one. Then
\begin{align}
\sum_{i=2}^{n-1} \displaystyle \min_{{i < j \leq n}}L_{ij} & \stackrel{d}{=} \sum_{i=2}^{n-1} \displaystyle \min_{{i < j \leq n}} F^{-1}(1-e^{-\xi_{ij}}) \nonumber \\
&\stackrel{d}= \displaystyle \sum_{i=2}^{n-1} F^{-1}(1-e^{-\displaystyle \min_{{i < j \leq n}}\xi_{ij}}) \nonumber \\
& \stackrel{d}{=}\displaystyle \sum_{i=1}^{n-2} F^{-1}(1-e^{-\frac{Y_{i}}{i}}) \nonumber 
\end{align}
where $Y_{i}$'s are i.i.d. Exponential  random variable each with mean one.

Finally \eqref{Tnd=W} follows from equation~\eqref{Tnd}.
\end{proof}

In the proofs of our main results, we primarily study properties of $W_{i}$ rather than $\displaystyle \min_{{i < j \leq n}}L_{ij}$. Observe that
\begin{equation}
\Pro(W_i \leq w)= 1-\lbrace 1-F(w)\rbrace^{i} \,\,\, \text{for}\,\,\, w\geq 0.
\label{dist of Wi}
\end{equation}

\begin{lemma}
\label{MG}
Assume that $F$ has a density $f$ and as $t\longrightarrow 0+,~ \frac{f(t)}{t^{\alpha}}\longrightarrow C$, 
where $C\in (0,\infty)$ is constant and  $-1< \alpha < 1$.  Then as $n\longrightarrow \infty,$ 
$
\lbrace\displaystyle \sum_{i=1}^{n-2}(W_{i}-\bE[W_{i}])\rbrace_{n\geq1},
 \label{MG-sumW} $
 converges $a.s.$ and in $\Lt.$
\end{lemma} 

\begin{proof}
By assumption as $t\longrightarrow 0+, \, \frac{f(t)}{t^{\alpha}}\longrightarrow C$, therefore given $\epsilon > 0$, there exists $\delta>0$, 
such that for all $ 0<t<\delta$, we have 
\[
(C-\epsilon)t^{\alpha} < f(t) < (C+\epsilon)t^{\alpha} \,.
\]
Hence for $0<x<\delta$,
\[  \frac{(C-\epsilon)}{1+\alpha} x^{1+\alpha} < F(x) < \frac{(C+\epsilon)}{1+\alpha} x^{1+\alpha}  \] 
which implies
\begin{equation}
 (\frac{1+\alpha}{C+\epsilon})^{{\frac{1}{1+\alpha}}} x^{{\frac{1}{1+\alpha}}} < 
F^{-1}(x) < (\frac{1+\alpha}{C-\epsilon})^{{\frac{1}{1+\alpha}}} x^{{\frac{1}{1+\alpha}}}. 
\label{ineq for F^-1} 
 \end{equation}

Put $\delta_1:=-\ln (1-\delta)$. If $\frac{Y_{i}}{i} < \delta_1 $ (which ensures that $1-\exp({-\frac{Y_{i}}{i}}) < \delta$), then we have
\begin{equation}  W_{i}~\bm{1}\left[ \frac{Y_{i}}{i}<\delta_1\right]  <(\frac{1+\alpha}{C-\epsilon})^{{\frac{1}{1+\alpha}}} \left( 1-\exp({-\frac{Y_{i}}{i}})\right) ^{{\frac{1}{1+\alpha}}}~\bm{1}\left[ \frac{Y_{i}}{i}<\delta_1\right] . \label{upper bdd W}\end{equation}
Observe that for $\beta>0$,
\begin{align}
\bE\left[ \left( 1-\exp({-\frac{Y_{i}}{i}})\right) ^\beta\right] &=\int_{0}^{\infty} (1-\exp(-y/i))^\beta \exp(-y) dy \nonumber \\
&=i \int_{0}^{1} u^\beta (1-u)^{i-1} du \nonumber\\
&=\Gamma(1+\beta)\frac{\Gamma(i+1)}{\Gamma(i+1+\beta)} \nonumber \\
&\leq \Gamma(2+\beta)\frac{1}{(i+1+\beta)^\beta}. \nonumber
\end{align}
The last inequality follows from the Wendel's double inequality \cite{We48}, which says for real $x>0$ and $0<s<1$ we have
\begin{equation} 
\frac{x}{(x+s)^{1-s}} \Gamma(x) \leq \Gamma(x+s) \leq x^s \Gamma(x) \label{Wendel's double inequality}
\end{equation}
Therefore
\begin{equation}
\bE\left[ W^2_{i}~\bm{1}[\frac{Y_{i}}{i}<\delta_1]\right] <(\frac{1+\alpha}{C-\epsilon})^{{\frac{2}{1+\alpha}}} \Gamma\left( 2+\frac{2}{1+\alpha}\right) \frac{1}{\left( i+1+\frac{2}{1+\alpha}\right) ^{\frac{2}{1+\alpha}}}. \label{upper bdd E(W^2)} \end{equation}
Now as $\displaystyle  i\longrightarrow \infty,~\frac{Y_{i}}{i} \stackrel{a.s.}\longrightarrow 0$. 
This follows from the Borel-Cantelli lemma, because for any $\epsilon_0 > 0$, the sequence of probabilities $\bP\left(Y_i > \epsilon_0 \, i\right) = e^{- \epsilon_0 \, i}$ are summable.  
Define \begin{equation}I_{0}(\omega):=min\left\{i ~\vert~ \frac{Y_{j}(\omega)}{j}<\delta_1, ~\forall j\geq i\right\}. \label{I_{0}}\end{equation}
Fix $m > 1$, then
\[[I_{0}=m]=\left[ \frac{Y_{i}}{i}<\delta_1, \forall i\geq m \quad \text{and} \quad \frac{Y_{m-1}}{m-1} >\delta_1\right] .\]
Hence
\[\mathbb{P}(I_{0}=m)\leq e^{-(m-1)\delta_1}\]
Now,
\[\sum_{i=1}^{\infty} \bE[W_{i}^{2}] =\sum_{m=1}^{\infty} \bE[\sum_{i=1}^{m-1}W_{i}^{2}~\bm{1}(I_{0}=m)]+\sum_{m=1}^{\infty}  \bE[\sum_{i=m}^{\infty}W_{i}^{2}~\bm{1}(I_{0}=m)].\]
But,
\[
\bE[\sum_{i=1}^{m-1}W_{i}^{2}~\bm{1}(I_{0}=m)]= \bE[\sum_{i=1}^{m-2}W_{i}^{2}~\bm{1}(I_{0}=m)]+ \bE[W_{m-1}^{2}\bm{1}(I_{0}=m)] \,.
\]
Since $[I_{0}=m]$ depends on random variables $Y_{m-1}, Y_{m}, Y_{m+1},...$ therefore for $1 \leq i \leq m-2$, $W_i$ is independent of $[I_{0}=m]$, hence  
 \[ \bE[\sum_{i=1}^{m-2}W_{i}^{2}~\bm{1}(I_{0}=m)] \leq e^{-(m-1)\delta_1}\sum_{i=1}^{m-2} \bE[ W_{i}^{2}].  \] 
Since $\bE[W_i^2]$ is a decreasing sequence, we have
\[\sum_{i=1}^{m-2} \bE[ W_{i}^{2}]\leq (m-2) \bE[ W_{1}^{2}]. \]
Therefore
\begin{equation} 
\bE[\sum_{i=1}^{m-2}W_{i}^{2}~\bm{1}(I_{0}=m)] \leq (m-2)e^{-(m-1)\delta_1}\bE[ W_{1}^{2}] \,. 
\label{BdW1}
\end{equation}

By Cauchy-Schwarz Inequality
\[
\bE[W_{m-1}^{2}\bm{1}(I_{0}=m)]] \leq \sqrt{\bE[W_{m-1}^{4}] \mathbb{P}(I_{0}=m)}.
\] 
Now for $m>4$,
\begin{equation}
\bE[W_{m-1}^{4}] \leq \bE[W_{4}^{4}] \leq \mu^{4}.
\label{W4<mu}
\end{equation}
Therefore
\begin{equation}
\bE[W_{m-1}^{2}\bm{1}(I_{0}=m)]] \leq \mu^2 e^{-(m-1)\frac{\delta_1}{2}}.\label{BdW2}\\
\end{equation} 
\linebreak
In the last equality of \eqref{W4<mu}, we use the fact that for $k$ non-negative random variables $Z_1, Z_2,..., Z_k$,
\[ \left( \min(Z_1,Z_2,...,Z_k) \right)^k \leq  \prod_{j=1}^{k}Z_{j}.\]
From \eqref{BdW1} and \eqref{BdW2},  we have
\begin{equation}\sum_{m=1}^{\infty} \bE[\sum_{i=1}^{m-1}W_{i}^{2}~\bm{1}(I_{0}=m)] < \infty \label{1st term is finite}\end{equation}
Now by assumption since $\vert \alpha \vert < 1$, we have $\frac{2}{1+\alpha} > 1$, therefore for $i\geq m$ from inequality~\eqref{upper bdd E(W^2)} we have
\begin{equation}\bE[\sum_{i=m}^{\infty}W_{i}^{2}~\bm{1}(I_{0}=m)] < K e^{-(m-1)\delta_1} \label{2ed term is finite}\end{equation}
where $K$ is a positive constant. Hence from \eqref{1st term is finite} and \eqref{2ed term is finite} we conclude
\begin{equation}
\sum_{i=1}^{\infty} \bE[W_{i}^{2}]  < \infty  \label{final} 
\end{equation}
 Therefore $\displaystyle \Var[\sum_{i=1}^{n}W_{i}]$ is bounded for all $n$. This shows that $\displaystyle \sum_{i=1}^{n-2}(W_{i}-\bE[W_{i}])$ as 
a martingale converges $a.s.$ and in $\mathcal{L}_{2}.$
\end{proof}

The following lemma gives an expression for the mean of $\TNN$ in terms of the distribution function $F$. Under some further assumption on $F$ it also shows 
how the behavior of $\bE\left[\TNN\right]$ depends on the behavior of the density $f$ of $F$ near zero.
\begin{lemma}
\label{expand E[TNN]}
Consider a mean field TSP with i.i.d. edge weights with distribution $F$ which is supported on $[0,\infty)$. Then
\[
\bE[\TNN]= \int_{0}^{\infty} \frac{\left[ \bar{F}(t)\right] ^2 \left[ 1-\left( \bar{F}(t)\right) ^{n-2}\right] }{F(t)} dt + 
\bE[L_n^{\mbox{first}} + L_n^{\mbox{last}}] \,.
\]
Moreover if $F$ admits a continuous density $f$ which is strictly positive on the support $[0,\infty)$ then 
\[
\bE[\TNN]= \int_{0}^{1} \frac{(1-w)^2 (1-[1-w]^{n-2})}{w} \frac{1}{f(F^{-1}(w))} dw + \bE[L_n^{\mbox{first}} + L_n^{\mbox{last}}]\,. 
\]
\end{lemma}

\begin{proof}
Let $\bar{F}(t)=1-F(t)$. From equation~\eqref{Tnd} we have 
\[\bE[\TNN]= \displaystyle\sum^{n-1}_{i=2} \bE[\displaystyle \min_{{i < j \leq n}}L_{ij}] + \bE[L_n^{\mbox{first}} + L_n^{\mbox{last}}]
\] 
But, \[ \bE[\displaystyle \min_{{i < j \leq n}}L_{ij}]=\int_{0}^{\infty} [\bar{F}(t)]^{n-i} dt ,\] and hence
\[
\bE[\TNN]= \int_{0}^{\infty} \frac{\left[ \bar{F}(t)\right] ^2 \left[ 1-\left( \bar{F}(t)\right) ^{n-2}\right] }{F(t)} dt + 
\bE[L_n^{\mbox{first}} + L_n^{\mbox{last}}] \,,
\]  
which proves the first part of the lemma.

Now if we assume that $F$ admits a continuous density $f$ which is strictly positive on the support $[0,\infty)$ then the second expression 
follows by changing the variable $w = F\left(t\right)$ in the first. 
\end{proof}

\section{Proofs of the main results}
\label{Proofs}

\subsection{Proof of Theorem~\ref{TNN-E(TNN) converges}}
\begin{proof}
From equation~\eqref{Tnd} we have
\begin{equation*}
\TNN-\bE[\TNN]\stackrel{d}{=}\sum_{i=2}^{n-1}  \min_{{i < j \leq n}}L_{ij}-\bE[\sum_{i=2}^{n-1}  
\min_{{i < j \leq n}}L_{ij}]+ L_n^{\mbox{first}} + L_n^{\mbox{last}}-\bE[L_n^{\mbox{first}} + L_n^{\mbox{last}}]. \label{Tn-E}
\end{equation*}
But by Lemma~\ref{MG} and Lemma~\ref{min L= min W},
$\left\lbrace \displaystyle\sum_{i=2}^{n-1}  \min_{{i < j \leq n}}L_{ij}-\bE[\displaystyle\sum_{i=2}^{n-1}  \min_{{i < j \leq n}}L_{ij}]\right\rbrace_{n>1}$ 
converges in $\Lt$ and hence by Proposition \ref{Prop1}, $\left\lbrace \TNN-\bE[\TNN]\right\rbrace _{n>1}$ converges weakly.
\end{proof}

\subsection{Proof of Theorem~\ref{Thm f(t)-->f(o)}}
\begin{proof}
We will show 
\[\frac{\TNN}{\log n} \stackrel{\Lt}{\longrightarrow} \frac{1}{f(0)} \quad \text{as} \quad n \longrightarrow \infty \,,\]
which will imply \eqref{Tn/logn-->1/f(o)}. Now,
\begin{align}
\bE\left[ \frac{\TNN}{\log n}-\frac{1}{f(0)}\right] ^2 &= \bE\left[ \frac{\TNN-\bE[\TNN]}{\log n} + \frac{\bE[\TNN]}{\log n}-\frac{1}{f(0)}\right] ^2 \nonumber\\
&=\frac{ \bE\left[ \left(\displaystyle\sum_{i=2}^{n-1}  \min_{{i < j \leq n}}L_{ij}-\bE[\displaystyle\sum_{i=2}^{n-1}\min_{{i < j \leq n}}L_{ij}]\right) ^2\right] }{(\log n)^2} \nonumber\\
&+\frac{\bE\left[\left(  L_n^{\mbox{first}} + L_n^{\mbox{last}}-\bE[L_n^{\mbox{first}} + L_n^{\mbox{last}}]\right) ^2\right] }{(\log n)^2}+ \left[ \frac{\bE[\TNN]}{\log n}-\frac{1}{f(0)}\right] ^2 \nonumber\\
&=\frac{ \Var\left[ \displaystyle\sum_{i=2}^{n-1}  \min_{{i < j \leq n}}L_{ij}\right]}{(\log n)^2} + \frac{ \Var\left[ L_n^{\mbox{first}} + L_n^{\mbox{last}}\right]}{(\log n)^2}\nonumber\\
&+ \left[ \frac{\bE[\TNN]}{\log n}-\frac{1}{f(0)}\right] ^2 .
 \label{L2 convergence of TNN/logn-1/f(0)} 
\end{align}
Note that $\displaystyle\sum_{i=2}^{n-1}  \min_{{i < j \leq n}}L_{ij}$ is independent of $L_n^{\mbox{last}}+L_n^{\mbox{first}}$. 
Now by  Lemma~\ref{MG}, Lemma~\ref{min L= min W} and Proposition~\ref{Prop1}, the first two terms in equation~\eqref{L2 convergence of TNN/logn-1/f(0)} 
converges to zero as $n \longrightarrow \infty$. Convergence to zero of the last term in equation~ \eqref{L2 convergence of TNN/logn-1/f(0)} follows from the 
following observation.  
By assumption $f(t)\longrightarrow f(0)$ as $t \longrightarrow 0+$, so using the inequality~\eqref{ineq for F^-1} when $f(0)=C$ and $\alpha=0$, we get that as $i\longrightarrow \infty$, \[\frac{f(0)W_{i}}{\frac{Y_{i}}{i}}\longrightarrow 1 \quad a.s.\]
where $Y_i$'s are i.i.d. Exponential random variable each with mean one and $\displaystyle W_{i}=F^{-1}\left( 1-\exp({-\frac{Y_{i}}{i}})\right) $. Therefore as $n\longrightarrow \infty$
\[\frac{f(0) \displaystyle \sum_{i=1}^{n-2}W_{i}}{\displaystyle \sum_{i=1}^{n-2}\frac{Y_{i}}{i}} \longrightarrow 1 \quad a.s.\]
Now, since $\Var \left[ \displaystyle \sum_{i=1}^{n-2}\frac{Y_{i}}{i}\right]$ is bounded for all $n$, therefore by the martingale convergence theorem   $\displaystyle \sum_{i=1}^{n-2}\frac{Y_{i}}{i}-\bE \left[ \displaystyle \sum_{i=1}^{n-2}\frac{Y_{i}}{i}\right] $ converges almost surely. But $\bE \left[ \displaystyle \sum_{i=1}^{n}\frac{Y_{i}}{i}\right]=\displaystyle \sum_{i=1}^{n} \frac{1}{i} \sim \log n$, thus
\begin{equation}\frac{f(0) \displaystyle \sum_{i=1}^{n-2}W_{i}}{\log n} \longrightarrow 1 \quad a.s. \label{part1}\end{equation}
Now by Lemma~\ref{MG} and Lemma~\ref{min L= min W},  $\displaystyle \sum_{i=1}^{n-2}W_{i} - \bE[\displaystyle \sum_{i=1}^{n-2}W_{i}]$ converges $a.s.$ to a random variable. This observation along with \eqref{part1} give
\begin{equation}\lim_{n \longrightarrow \infty} \frac{\bE[\displaystyle \sum_{i=1}^{n-2}W_{i}]}{\log n}=\frac{1}{f(0)}  \label{Lim}\end{equation}
and therefore by equation~\eqref{Tnd=W} and Proposition~\ref{Prop1},
\[
\lim_{n \longrightarrow \infty} \frac{\bE[\TNN]}{\log n}=\frac{1}{f(0)} \,.  
\]
This also proves $\bE[\TNN] \sim \frac{1}{f(0)}\log n$.
\end{proof}

\subsection{Proof of Corollary~\ref{Thm for iid exp}}
\begin{proof}
Consider a mean field TSP on $n$ cities $\left\{1, 2, ..., n\right\}$, where for each $1\leq i \leq n-1$, the intercity distances $\left\{L_{ij}\right\}_{i< j \leq n}$, are i.i.d. Exponential random variable each with mean one. Starting at city 1, our job is to find the nearest city to it, that means to find $\displaystyle \min_{{1 < j \leq n}}L_{1j}$.
Now we have a tour, with 2 cities in it. Finding the next nearest city to the last visited city in this tour, in distribution is the same as finding the \textit{minimum} of $n-3$ independent Exponential random variables. \\
Since $\displaystyle \min_{{i < j \leq n}}L_{ij}$ has an Exponential distribution with mean $\frac{1}{n-i}$, then we have
\begin{equation} \bE [\sum_{i=1}^{n-1}  \min_{{i < j \leq n}}L_{ij}]=\frac{1}{n-1} + \frac{1}{n-2} + \ldots + \frac{1}{2} + 1 \label{E(min)=sum 1/i}\end{equation}
Since $ \displaystyle \Var[\displaystyle\sum_{i=1}^{n-1}  \min_{{i < j \leq n}}L_{ij}] = \displaystyle\sum_{i=1}^{n-1}\frac{1}{i^{2}}$, hence for all $n\geq 1, \Var\left(\!\displaystyle\sum_{i=1}^{n-1}\!\min_{{i < j \leq n}}L_{ij}-\bE[\!\displaystyle\sum_{i=1}^{n-1}\!\min_{{i < j \leq n}}L_{ij}]\right)  $ is bounded. Therefore by the martingale convergence theorem, we conclude that the martingale sequence
\begin{equation}
\left\lbrace \sum_{i=1}^{n-1}  \min_{{i < j \leq n}}L_{ij}-\bE[\sum_{i=1}^{n-1}  \min_{{i < j \leq n}}L_{ij}]\right\rbrace _{n \geq 1} \quad \text{converges} \quad a.s. \quad \text{and in} \quad \Lt.
\label{e1}
\end{equation} 
Note that as we saw in equation~\eqref{E(min)=sum 1/i},  $\displaystyle\bE[\sum_{i=1}^{n-1}  \min_{{i < j \leq n}}L_{ij}]=\displaystyle\sum_{i=1}^{n-1} \frac{1}{i}$. Using the fact that,
\[\displaystyle\sum_{i=1}^{n} \frac{1}{i} = \log n +\gamma +O(\frac{1}{n})\, \]
where $\gamma\!:=\!\displaystyle\lim_{n\longrightarrow\infty}\!\left(\!\sum_{k=1}^{n}\!\frac{1}{k}-\log n\!\right)$ is the Euler constant,
shows that  
$\displaystyle{\!\left\lbrace\bE[\!\sum_{i=1}^{n-1}\!\min_{{i < j \leq n}}L_{ij}]-\log n\!\right\rbrace_{n\!\geq\!1}}$ is a convergent sequence.
Now from \eqref{Tnd}, we have
\begin{align*}
\TNN-\log n &\stackrel{d}{=} \sum_{i=2}^{n-1}  \min_{{i < j \leq n}}L_{ij}-\bE[\sum_{i=2}^{n-1}  \min_{{i < j \leq n}}L_{ij}]+\bE[\sum_{i=2}^{n-1}  \min_{{i < j \leq n}}L_{ij}]-\log n + L_n^{\mbox{first}} + L_n^{\mbox{last}}\\
&\stackrel{d}{=} \sum_{i=2}^{n-1}  \min_{{i < j \leq n}}L_{ij}-\bE[\sum_{i=2}^{n-1}  \min_{{i < j \leq n}}L_{ij}]+\bE[\sum_{i=1}^{n-1}  \min_{{i < j \leq n}}L_{ij}]-\log n \\
&+ L_n^{\mbox{first}}+ L_n^{\mbox{last}}-\bE\left[L_n^{\mbox{first}}\right].
\end{align*}
Therefore by using \eqref{e1} and Proposition~\ref{Prop1}, we get $\left(  \TNN-\log n\right)  _{n\geq1}$ converges weakly.  
\end{proof}

\subsection{Proof of Theorem~\ref{Thm f(t)/t^alpha, 0<alpha<1}}
\begin{proof}
Recall the double inequality~\eqref{ineq for F^-1} in the proof of Lemma~\ref{MG}. By the assumption of the theorem and ~\eqref{ineq for F^-1}, as $i\longrightarrow \infty$, \[\frac{(\frac{C}{1+\alpha})^\frac{1}{1+\alpha}W_{i}}{(\frac{Y_{i}}{i})^\frac{1}{1+\alpha}}\longrightarrow 1 \quad a.s.\]
where $Y_i$'s are i.i.d. Exponential random variable each with mean one and $\displaystyle W_{i}=F^{-1}\left( 1-\exp({-\frac{Y_{i}}{i}})\right) $. Therefore as $n\longrightarrow \infty$
\[\frac{(\frac{C}{1+\alpha})^\frac{1}{1+\alpha} \displaystyle \sum_{i=1}^{n-2}W_{i}}{\displaystyle \sum_{i=1}^{n-2}(\frac{Y_{i}}{i})^{\frac{1}{1+\alpha}}} \longrightarrow 1 \quad a.s.\]
Since $0 < \alpha < 1$ so $\frac{2}{1+\alpha} > 1$, thus  $\Var\left(\displaystyle \sum_{i=1}^{n-2}(\frac{Y_{i}}{i})^{\frac{1}{1+\alpha}}\right) $ is uniformly bounded and so by the martingale convergence theorem  $\displaystyle \sum_{i=1}^{n-2}(\frac{Y_{i}}{i})^{\frac{1}{1+\alpha}}-\bE \left[ \displaystyle \sum_{i=1}^{n-2}(\frac{Y_{i}}{i})^{\frac{1}{1+\alpha}}\right] $ converges almost surely. But \[\bE \left[ \displaystyle \sum_{i=1}^{n-2}(\frac{Y_{i}}{i})^{\frac{1}{1+\alpha}}\right]=\Gamma(1+\frac{1}{1+\alpha})\displaystyle \sum_{i=1}^{n-2}(\frac{1}{i})^{\frac{1}{1+\alpha}}\,.\]
Thus 
\begin{equation}\frac{\displaystyle \sum_{i=1}^{n-2}W_{i}}{K_{\alpha}n^{1-\frac{1}{1+\alpha}}} \longrightarrow 1 \quad a.s. \label{Wi/n-->C(a)} \end{equation}
where
\[K_{\alpha}:=(\frac{1+\alpha}{C})^{\frac{1}{1+\alpha}} \frac{1+\alpha}{\alpha}\Gamma(1+\frac{1}{1+\alpha})\,.\] Now 
\[\displaystyle \sum_{i=1}^{n-2}W_{i}- K_{\alpha} n^{1-\frac{1}{1+\alpha}}= \displaystyle \sum_{i=1}^{n-2}W_{i}-\bE[\displaystyle \sum_{i=1}^{n-2}W_{i}]+ \bE[\displaystyle \sum_{i=1}^{n-2}W_{i}]-K_{\alpha}n^{1-\frac{1}{1+\alpha}} \,.\]
Recall that by Lemma~\ref{MG},
$\displaystyle \sum_{i=1}^{n-2}W_{i}-\bE[\displaystyle \sum_{i=1}^{n-2}W_{i}]$ has an almost sure limit, so using  \eqref{Wi/n-->C(a)} we get
\begin{equation}
\lim_{n \longrightarrow \infty} \frac{\bE[\displaystyle \sum_{i=1}^{n-2}W_{i}]}{n^{1-\frac{1}{1+\alpha}}}=K_{\alpha}
\end{equation}
and hence by Lemma~\ref{MG}, Lemma~\ref{min L= min W} and equation~\eqref{Tnd=W},
\[\bE[\TNN] \sim K_{\alpha} n^{1-\frac{1}{1+\alpha}} .\]
Note that
 \begin{align*}
\bE[\frac{\TNN} {n^{1-\frac{1}{1+\alpha}}}-K_{\alpha}]^{2}&=\bE[\frac{\TNN-\bE[\TNN]} {n^{1-\frac{1}{1+\alpha}}}+\frac{\bE[\TNN]} {n^{1-\frac{1}{1+\alpha}}}-K_{\alpha}]^{2}\nonumber\\
&=\frac{ \bE\left[ \left(\displaystyle\sum_{i=2}^{n-1}  \min_{{i < j \leq n}}L_{ij}-\bE[\displaystyle\sum_{i=2}^{n-1}\min_{{i < j \leq n}}L_{ij}]\right) ^2\right] }{(n^{1-\frac{1}{1+\alpha}})^2} \nonumber\\
&+\frac{\bE\left[\left(  L_n^{\mbox{first}} + L_n^{\mbox{last}}-\bE[L_n^{\mbox{first}} + L_n^{\mbox{last}}]\right) ^2\right] }{(n^{1-\frac{1}{1+\alpha}})^2}+ \left[ \frac{\bE[\TNN]}{n^{1-\frac{1}{1+\alpha}}}-K_{\alpha}\right] ^2 \nonumber\\
&=\frac{ \Var\left[ \displaystyle\sum_{i=2}^{n-1}  \min_{{i < j \leq n}}L_{ij}\right]}{(n^{1-\frac{1}{1+\alpha}})^2} + \frac{ \Var\left[ L_n^{\mbox{first}} + L_n^{\mbox{last}}\right]}{(n^{1-\frac{1}{1+\alpha}})^2}\nonumber\\
&+ \left[ \frac{\bE[\TNN]}{n^{1-\frac{1}{1+\alpha}}}-K_{\alpha}\right] ^2
\end{align*}
converges to zero as $n \longrightarrow \infty$. Hence
\[
\frac{\TNN} {n^{1-\frac{1}{1+\alpha}}}\stackrel{\Pro}{\longrightarrow}K_{\alpha}
\]
and in $\Lt$.
\end{proof}

\subsection{Proof of Theorem \ref{Thm f(t)/t^alpha, -1<alpha<0}}
\begin{proof}
As it has mentioned in the proof of Lemma~\ref{MG}, since ${\frac{1}{1+\alpha}}>1$, we get
\[ \sup_{n\geq1}\Var(\displaystyle \sum_{i=1}^{n-2}W_{i})<\infty \,.\]
Therefore $\displaystyle \sum_{i=1}^{n-2}W_{i}-\bE[\displaystyle \sum_{i=1}^{n-2}W_{i}]$ as a martingale converges $a.s.$ and in $\Lt$.
So by equation~\eqref{Tnd=W} and Proposition~\ref{Prop1}, $\TNN-\bE[\TNN]$ converges weakly. 

Now to complete the proof it is enough to show that $\left\{\bE\left[\TNN\right]\right\}_{n \geq 1}$ is a convergent sequence. For that 
we apply Lemma~\ref{expand E[TNN]} to get
\begin{equation}
\bE[\TNN]= \int_{0}^{\infty} \frac{\left[ \bar{F}(t)\right] ^2 \left[ 1-\left( \bar{F}(t)\right) ^{n-2}\right] }{F(t)} dt + 
\bE[L_n^{\mbox{first}} + L_n^{\mbox{last}}] \,.
\label{Equ:Expect-TNN}
\end{equation}
Now fix $\epsilon > 0$ and get $\delta > 0$ such that the equations leading to the
double inequality~\eqref{ineq for F^-1} holds. Also find $M > 0$ such that $F\left(M\right) \geq \frac{1}{2}$. Consider the function
$G: [0,\infty) \rightarrow [0, \infty)$ defined as
\[
G\left(t\right) := \left\{
                   \begin{array}{cl}
                   \frac{1}{F(t)} & \mbox{if\ } 0 < t < \delta \\
                   \frac{1}{F(\delta)} & \mbox{if\ } \delta \leq t \leq M \\
                   2 \bar{F}(t) & \mbox{otherwise}
                   \end{array} 
                   \right. \,.
\]
Then for any $n > 1$ and $t > 0$ we have
\[
\frac{\left[ \bar{F}(t)\right] ^2 \left[ 1-\left( \bar{F}(t)\right) ^{n-2}\right] }{F(t)} \leq  G(t) \,.
\]
Also note that $\int_M^{\infty} \! G(t) \, dt \leq 2 \int_0^{\infty} \! \bar{F}(t) \, dt < \infty$ as $F$ is positively supported and has finite first moment. 
Further by the choice of $\delta$ we get that on $\left(0,\delta\right)$ the density $f$ is strictly positive and $F$ is strictly increasing. So
\begin{align*}
       \int_0^{\delta} \! G(t) \, dt 
& =    \int_0^{\delta} \! \frac{dt}{F(t)} \\
& =    \int_0^{F(\delta)} \! \frac{dw}{w \, f\left(F^{-1}(w)\right)} \qquad \left[\mbox{substitute\ } w=F(t)\right] \\
& \leq \kappa \int_0^1 \! \frac{1}{w^{1+\frac{\alpha}{1+\alpha}}} \, dw  < \infty \,,
\end{align*}
where $\kappa > 0$ is some constant and the last but one inequality follows by using the double inequality~\eqref{ineq for F^-1} and the final inequality holds 
because $-1 < \alpha < 0$. 
Thus we get that 
\[
\int_0^{\infty} \! G(t) \, dt < \infty \,.
\]
So by the dominated convergence theorem we conclude that
\[
\lim_{n \rightarrow \infty} \int_{0}^{\infty} \frac{\left[ \bar{F}(t)\right] ^2 \left[ 1-\left( \bar{F}(t)\right) ^{n-2}\right] }{F(t)} dt 
\]
exists.
This along with Proposition~\ref{Prop1} proves that $\left\{ \bE\left[\TNN\right] \right\}_{n \geq 1}$ is convergent sequence, which completes the proof of the
theorem. 
\end{proof}

\subsection{Proof of Theorem \ref{Prop f(t)/t^alpha,alpha>=1}}
\begin{proof}
By assumption that $F$ is Weibull distribution with shape parameter $(1+\alpha)$ and scale parameter $1$, we get \[F(x)=1-e^{-x^{1+\alpha}}, \quad x \geq 0 \] 
Therefore $F^{-1}(t)=[-\log (1-t)]^{\frac{1}{1+\alpha}}$, where $ 0 < t< 1$. Hence,
\begin{align*}
 \sum_{i=2}^{n-1} \displaystyle \min_{{i < j \leq n}}L_{ij} &\stackrel{d}{=}\sum_{i=1}^{n-2} W_i \\
 & =\sum_{i=1}^{n-2}[-\log (e^{-\frac{Y_i}{i}})]^{\frac{1}{1+\alpha}} \\
 & =\sum_{i=1}^{n-2}(\frac{Y_{i}}{i})^{\frac{1}{1+\alpha}} 
\end{align*}
where $Y_i$'s are i.i.d. Exponential random variable each with mean one. Note that
\[\mu(\alpha):=\mathbb{E}\left[ Y_{i}^{\frac{1}{1+\alpha}}\right] =\Gamma(1+\frac{1}{1+\alpha})\]
and
\[\sigma^2(\alpha):=\Var\left[ Y_{i}^{\frac{1}{1+\alpha}}\right] =\Gamma(\frac{2}{1+\alpha}+1)-\Gamma^{2}(1+\frac{1}{1+\alpha})\, .\]
Let \[V_{i}(\alpha):=\frac{Y_{i}^{\frac{1}{1+\alpha}}-\mathbb{E}[Y_{i}^{\frac{1}{1+\alpha}}]}{\sigma(\alpha)  i^{\frac{1}{1+\alpha}}\sqrt{\displaystyle \sum_{i=1}^{n-2}(\frac{1}{i})^{\frac{2}{1+\alpha}}}}\]
and $Z_{n}(\alpha)=\displaystyle \sum_{i=1}^{n-2}V_{i}(\alpha)$. Observe that 
$\mathbb{E}[V_{i}(\alpha)]=0$ and $\displaystyle \sum_{i=1}^{n-2}\Var[V_{i}(\alpha)]=1\,$. Choose $\delta > 0$ such that $\delta > \alpha-1$. So for some $M>0$,
\[\displaystyle \sum_{i=1}^{n-2}\mathbb{E}\left[ |V_{i}(\alpha)|^{2+\delta}\right] \leq \frac{M}{\sigma(\alpha)^{2+\delta}}\frac{1}{[\displaystyle \sum_{i=1}^{n-2}(\frac{1}{i})^{\frac{2}{1+\alpha}}]^{\frac{2+\delta}{2}}} \displaystyle \sum_{i=1}^{n-2}(\frac{1}{i})^{\frac{2+\delta}{1+\alpha}}\,.\]
Since $\frac{2}{1+\alpha}\leq 1$ and $\frac{2+\delta}{1+\alpha} > 1$, we have
\[\lim_{n\rightarrow\infty}\displaystyle \sum_{i=1}^{n-2}\mathbb{E}\left[ |V_{i}(\alpha)|^{2+\delta}\right] =0 \, .\]
Hence Lyapunov condition is satisfied for $\alpha \geq 1$ and so $Z_n(\alpha)$ converges in distribution to a standard Normal random variable, as $n$ goes to infinity. Now by equation~\eqref{Tnd} we have
\begin{align*}
\frac{\TNN-\bE[\TNN]}{n^{\frac{1}{2}-\frac{1}{1+\alpha}}} &\stackrel{d}{=}\frac{\displaystyle\sum_{i=1}^{n-2}(\frac{Y_{i}}{i})^{\frac{1}{1+\alpha}}-\bE[\displaystyle\sum_{i=1}^{n-2}(\frac{Y_{i}}{i})^{\frac{1}{1+\alpha}}]}{\lbrace \Var[\displaystyle\sum_{i=1}^{n-2}(\frac{Y_{i}}{i})^{\frac{1}{1+\alpha}}]\rbrace^{1/2}} \frac{\lbrace \Var[\displaystyle\sum_{i=1}^{n-2}(\frac{Y_{i}}{i})^{\frac{1}{1+\alpha}}]\rbrace^{1/2}}{n^{\frac{1}{2}-\frac{1}{1+\alpha}}}\\
&+ \frac{L_n^{\mbox{first}} + L_n^{\mbox{last}}-\bE[L_n^{\mbox{first}} + L_n^{\mbox{last}}]}{n^{\frac{1}{2}-\frac{1}{1+\alpha}}}\,,\end{align*}
and thus the proof of proposition for $\alpha >1 $ is completed by Proposition~\ref{Prop1}. Note that when $\alpha=1$, by equation~\eqref{Tnd} we get
 \begin{align*}
\TNN-\bE[\TNN] &\stackrel{d}{=}\frac{\displaystyle\sum_{i=1}^{n-2}(\frac{Y_{i}}{i})^{\frac{1}{2}}-\bE[\displaystyle\sum_{i=1}^{n-2}(\frac{Y_{i}}{i})^{\frac{1}{2}}]}{\lbrace \Var[\displaystyle\sum_{i=1}^{n-2}(\frac{Y_{i}}{i})^{\frac{1}{2}}]\rbrace^{1/2}} \lbrace \Var[\displaystyle\sum_{i=1}^{n-2}(\frac{Y_{i}}{i})^{\frac{1}{2}}]\rbrace^{1/2} \\
&+ L_n^{\mbox{first}} + L_n^{\mbox{last}}-\bE[L_n^{\mbox{first}} + L_n^{\mbox{last}}]\,.\end{align*}
 But, 
 \[ \Var[\displaystyle\sum_{i=1}^{n-2}(\frac{Y_{i}}{i})^{\frac{1}{2}}]= \sigma^2(1)\displaystyle\sum_{i=1}^{n-2} \frac{1}{i}\]
 Therefore by Proposition~\ref{Prop1} and the fact that $\displaystyle\sum_{i=1}^{n-2} \frac{1}{i} \sim \log n$ we get, \[\frac{\TNN-\bE[\TNN]}{\sqrt{\log n}}\stackrel{d}{\longrightarrow}  N(0,\sigma^2(1)) \]
\end{proof}

\section{Discussion}
\label{discuss}
In our theorems, we assumed that the second moment of $F$ exists.
This assumption is not needed. The following lemma says that if $F$ is a positively supported distribution with finite $\beta^{\mbox{th}}$-moment then for any 
$k > \frac{2}{\beta}$ we must have $\bE\left[ \left(\mathop{\min}\limits_{1 \leq i \leq k} Z_i\right)^2\right] < \infty$ where $Z_1, Z_2, \ldots$ are i.i.d. $F$. 
  
\begin{lemma}
\label{higher moment}
Suppose $Z$ is a non-negative random variable such that for some $\beta > 0,\, \bE[Z^{\beta}] < \infty$. Then for any $k > \frac{2}{\beta}$ we have
\[ \int_{0}^{\infty} \! t \left\{ \mathbb{P}(Z>t) \right\}^{k} \, dt < \infty \,.\]
\end{lemma}
The proof of this lemma follows easily from Markov's inequality, so we omit it here. 
Now as before let random variable $\displaystyle  W_{i}=F^{-1}\left( 1-\exp({-\frac{Y_{i}}{i}})\right) $ where $Y_{i}$'s are Exponential with mean one. 
We have assumed $F$ has finite first moment so then by taking $k=3$ in Lemma \ref{higher moment} above we can conclude that $W_i$ has finite second moment 
for $i \geq 3$. Thus under the assumptions of Lemma \ref{MG} and following the proof of this lemma we can conclude that $\displaystyle{\sum_{i=k}^{n-2}(W_{i}-\bE[W_{i}])}$
converges almost surely and in $\Lt$. Thus all the results stated in Section \ref{Main results} hold except those on $\Lt$ convergence.

\bibliographystyle{plain}
\bibliography{reference}

\end{document}